\documentclass[11pt]{amsart}

\usepackage{ucs}\usepackage[utf8x]{inputenc}
\usepackage{amsfonts,amssymb,amsthm,amsmath,amsxtra,amscd,verbatim,eucal}
\usepackage{enumerate}
\usepackage[all]{xy}
\usepackage[dvips]{graphics}
\usepackage{psfrag}
\usepackage{hyperref}

\theoremstyle{plain}
\newtheorem{thm}{Theorem}[section]
\newtheorem*{theoremA}{Theorem A}
\newtheorem*{theoremB}{Theorem B}
\newtheorem*{theoremC}{Theorem C}
\newtheorem*{thm*}{Theorem}
\newtheorem*{prop*}{Proposition}

\newtheorem{lem}[thm]{Lemma}
\newtheorem{prop}[thm]{Proposition}
\newtheorem{cor}[thm]{Corollary}

\theoremstyle{definition}
\newtheorem{defi}[thm]{Definition}
\newtheorem{rmk}[thm]{Remark}
\newtheorem{eg}[thm]{Example}

\newtheorem{question}[thm]{Question}

\DeclareMathOperator{\Amp}{Amp}

\DeclareMathOperator{\im}{im}

\DeclareMathOperator{\ann}{ann}

\renewcommand{\P}{\ensuremath{\mathbb P}}

\newcommand{\NN}{\ensuremath{\mathbb N}}
\newcommand{\PP}{\ensuremath{\mathbb P}}
\newcommand{\QQ}{\ensuremath{\mathbb Q}}
\newcommand{\RR}{\ensuremath{\mathbb R}}
\newcommand{\ZZ}{\ensuremath{\mathbb Z}}
\newcommand{\CC}{\ensuremath{\mathbb C}}
\newcommand{\OO}{\ensuremath{\mathcal O}}
\newcommand{\II}{\ensuremath{\mathcal I}}
\newcommand{\F}{\ensuremath{\mathbb F}}

\newcommand\lra{\longrightarrow}
\newcommand\xred{X_{\textrm{red}}}

\newcommand{\shf}{\ensuremath{{\mathcal F}}}

\newcommand{\shn}{\ensuremath{{\mathcal N}}}
\newcommand{\shg}{\ensuremath{{\mathcal G}}}

\newcommand{\shl}{\ensuremath{{\mathcal L}}}

\newcommand{\shk}{\ensuremath{{\mathcal K}}}
\newcommand{\shc}{\ensuremath{{\mathcal C}}}
\newcommand{\shi}{\ensuremath{{\mathcal I}}}

\newcommand{\hh}[3]{\ensuremath{h^{#1}\left(#2,#3\right)}}

\newcommand{\hhat}[3]{\ensuremath{\widehat{h}^{#1}\left(#2,#3\right)}}
\newcommand{\HH}[3]{\ensuremath{H^{#1}\left(#2,#3\right)}}
\newcommand{\HHv}[3]{\ensuremath{H^{#1}(#2,#3)}}

\newcommand{\ses}[3]{\ensuremath{0\rightarrow #1 \rightarrow #2 \rightarrow #3 \rightarrow 0}}

\newcommand{\vol}[2]{\ensuremath{{\rm vol}_{#1}\left( #2 \right) } }

\newcommand{\st}[1]{\ensuremath{ \left\{ #1 \right\} }}

\newcommand{\deq}{\ensuremath{ \stackrel{\textrm{def}}{=}}}
\newcommand\R{\mathbb R}

\def\.{\cdot}
\def\^{\widehat}
\def\~{\widetilde}

\newcommand\newop[2]{\def#1{\mathop{\rm #2}\nolimits}}
\newop\voll{vol}
\newop\NS{NS}
\newop\Neg{Neg}
\newop\Null{Null}
\newop\Pic{Pic}
\newop\Bstab{B_{+}}
\newop\Bst{B_{stab}}
\newop\Bres{B_{restr}}
\newop\Bplus{\mathbf{B}_+}
\newop\Bminus{\mathbf{B}_-}
\newop\Exc{Exc}
\newop\B{\mathbf{B}{}}
\newop\Bs{Bs}
\newop\End{End}
\newop\Amp{Amp}
\newop\Face{Face}
\newop\BigCone{Big}
\newop\index{ind}
\newop\reg{reg}

\newcommand\eqnref[1]{(\ref{#1})}

\newcommand{\equ}{\ensuremath{\,=\,}}
\newcommand{\dleq}{\ensuremath{\,\leq\,}}
\newcommand{\dgeq}{\ensuremath{\,\geq\,}}
\newcommand{\dsup}{\ensuremath{\,\supset\,}}
\newcommand{\dsubseteq}{\ensuremath{\,\subseteq\,}}

\begin{document}


\title{Positivity  on subvarieties and vanishing of higher cohomology}
\dedicatory{To the memory of Eckart Viehweg}

\thanks{During this project the author was partially supported by  DFG-Forschergruppe 790 ``Classification of Algebraic Surfaces
and Compact Complex Manifolds'', and   the OTKA Grants 77476 and  77604 by the Hungarian Academy of Sciences.}

\author{Alex K\"uronya}
\address{Budapest University of Technology and Economics, Department of Algebra, P.O. Box 91, H-1521 Budapest, Hungary}
\address{Albert-Ludwigs-Universit\"at Freiburg, Mathematisches Institut, Eckerstra\ss e 1, D-79104 Frei\-burg, Germany}
\email{{\tt alex.kueronya@math.uni-freiburg.de}}



\maketitle

\section*{Introduction}

Inspired by the recent paper \cite{Totaro} of Totaro, we investigate the relationship between  ampleness of restrictions of line bundles to general complete intersections and the vanishing properties of higher cohomology groups. This train of thought will eventually lead us to a generalization of Fujita's vanishing theorem to big line bundles. 

Vanishing theorems  played a central role in algebraic geometry during the last fifty years.  Results of this sort due to Serre, Kodaira, Kawamata--Viehweg among others are fundamental building blocks of complex  geometry, and  are indispensable to the successes of minimal model theory. Classically, vanishing theorems apply to ample or big and nef line bundles. However, there has been a recent shift of attention towards big line bundles, which, although possess less positivity, still turn out to share many of the good properties of ample ones (see \cite{AIL} and the references therein). 

It has been known for some time that big line bundles behave in a cohomologically positive way  in degrees roughly  above the dimension of the stable base locus. An easy asymptotic version of this appeared in \cite[Proposition 2.15]{ACF}, while  Matsumura  \cite[Theorem 1.6]{Mats}  
gave a partial generalization of the Kawamata--Viehweg vanishing theorem along these lines.  In this work  we will present similar results guaranteeing the vanishing of cohomology groups of high degree under various partial positivity conditions. We work over the complex numbers, $X$ is an irreducible projective variety of dimension $n$ unless otherwise mentioned. Divisors are meant to be Cartier unless otherwise mentioned. 

Following the footsteps of Andreotti--Grauert~\cite{AG} and Demailly--Peternell--Schneider~\cite{DPS},  Totaro establishes a very satisfactory theory of line bundles with vanishing cohomology above  a certain degree. The concept has various  characterizations  all falling  under the heading of $q$-ampleness.  What is of  interest to us is the following version, which goes by  the name 'naive $q$-ampleness' in \cite{Totaro}.  We call a line bundle $L$ (naively) $q$-ample on $X$ for a natural number $q$, if for every coherent sheaf $\shf$ on $X$ there exists an 
integer $m_0$ depending on $\shf$ such that 
\[
 \HH{i}{X}{\shf\otimes \OO_X(mL)} \equ 0 \ \ \ \text{for all $i>q$ and $m\geq m_0$.}
\]
It is immediate from the definition that $0$-ampleness coincides with ampleness, while it is proved in \cite[Theorem 10.1]{Totaro} that a divisor is $(n-1)$-ample exactly if it does not lie in the negative  of the pseudo-effective cone in the N\'eron--Severi space. The notion of $q$-ampleness shares many important properties of traditional ampleness, for example it is open both in families and in the N\'eron--Severi space.
In general, the behaviour of $q$-ample divisors remains mysterious. 

Our motivation comes from the connection to geometric invariants describing partial positivity, in particular, to amplitude on restrictions to general complete intersection subvarieties.  The main results  of this note are   vanishing theorems valid for not necessarily ample --- oftentimes not even big --- divisors. They all follow the same principle:  positivity of restrictions of line bundles results in partial vanishing of higher cohomology groups.  Our  first  statement of note is a uniform variant of Serre vanishing. 

\begin{theoremA}(Theorem~\ref{thm:A})
Let $X$ be an irreducible projective variety, $L$ a Cartier divisor, $A_1,\dots,A_q$   very ample  Cartier divisors on $X$ such that $L|_{E_1\cap\dots\cap E_q}$ is ample for general $E_j\in |A_j|$, $1\leq j\leq q$. Then for any coherent sheaf $\shf$ on $X$ there exists an
integer $m(L,A_1,\dots,A_q,\shf)$ such that 
\[
 \HHv{i}{X}{\shf\otimes \OO_X(mL+\sum_{j=1}^{q}k_jA_j)} \equ 0
\]
for all $i>q$, $m\geq m(L,A_1,\dots,A_q,\shf)$ and $k_1,\dots,k_q\geq 0$. 
\end{theoremA}

In particular, setting $k_1=\dots=k_q=0$ in the Theorem  we obtain that  $L$ is $q$-ample. This way we recover a slightly weaker version of  \cite[Theorem 3.4]{DPS}. In a few remarks we then relate Theorem A  to invariants expressing partial positivity  and the inner structure of various  cones of divisors in the N\'eron--Severi space. Here again we go along similar lines to \cite{DPS}; it turns out that sacrificing a certain amount of generality buys  drastically simplified proofs. 

Next, we treat the case of vanishing for adjoint divisors. First, a variant  of the theorem of  Kawamata and Viehweg.

\begin{theoremB}(Theorem~\ref{thm:B})
 Let $X$ be a smooth projective variety, $L$ a divisor, $A$ a very ample divisor on $X$. If $L|_{E_1\cap\dots\cap E_k}$ is big and nef for a general choice of $E_1,\dots,E_k\in |A|$, then 
\[
 \HH{i}{X}{\OO_X(K_X+L)} \equ 0 \ \ \ \text{for $i>k$.}
\]
\end{theoremB}

The conditions of Matsumura hold true under our assumptions, hence we recover \cite[Theorem 1.6]{Mats}.

Building on the above result, we arrive at our main achievement,  a generalization of Fujita's theorem \cite{Fuj} to big divisors.

\begin{theoremC}(Theorem~\ref{thm:Fujita})
Let $X$ be a complex projective scheme, $L$ a Cartier divisor, $\shf$ a coherent sheaf on $X$. Then there exists a positive integer $m_0(L,\shf)$ such that 
\[
 \HH{i}{X}{\shf\otimes\OO_X(mL+D)} \equ 0
\]
 for all $i>\dim \Bplus (L)$, $m\geq m_0(L,\shf)$, and all nef divisors $D$ on $X$.
\end{theoremC}

Here $\Bplus(L)$ denotes the augmented base locus of $L$; this  can be defined as the stable base locus of the $\QQ$-Cartier divisor $L-A$ for any sufficiently small ample class $A$.

A few words about the organization of the paper. Section 1 is devoted to Theorem A, and a discussion of  invariants measuring partial positivity.
Theorem B is treated in Section 2, while Section 3 is given over to a short treatment of base loci on schemes. The proof of Theorem C takes up the last section. 

\subsection*{Acknowledgements} Helpful discussions with  Brian Conrad, Tommaso de Fernex, Lawrence Ein, Daniel Greb, Stefan Kebekus, Rob Lazarsfeld, Vlad Lazi\'c, Sebastian Neumann, Mihnea Popa, Tomasz Szemberg, and Burt Totaro were much appreciated. 


\section{Ampleness on restrictions and cones in the N\'eron--Severi space}

In this section we prove  a Fujita--Serre type vanishing statement  and consider an  application to  cone structures in $N^1(X)_\RR$.
This is where one can see most clearly  the yoga of obtaining partial vanishing of higher cohomology groups by forcing ampleness on restrictions.

\begin{thm}\label{thm:A}
 Let $X$ be an irreducible projective variety, $L$ a Cartier divisor, $A_1,\dots,A_q$   very ample  Cartier divisors on $X$ such that $L|_{E_1\cap\dots\cap E_q}$ is ample for general $E_j\in |A_j|$. Then for any coherent sheaf $\shf$ on $X$ there exists an
integer $m(L,A_1,\dots,A_q,\shf)$ such that 
\[
 \HHv{i}{X}{\shf\otimes \OO_X(mL+\sum_{j=1}^{q}k_jA_j)} \equ 0
\]
for all $i>q$, $m\geq m(L,A_1,\dots,A_q,\shf)$ and $k_j\geq 0$. In particular, $L$ is $q$-ample.
\end{thm}

\begin{proof}
Let $E_j\in |A_j|$ ($1\leq j\leq q$) be a sequence of general divisors. In particular  we assume all possible intersections  to be irreducible.   Consider the  set of standard exact sequences 
\begin{eqnarray}\label{eqn:ses}
 0 & \to & \shf\otimes\OO_{Y_j}(mL+\sum_{l=1}^{q}{k_lA_l}) \to  \shf\otimes\OO_{Y_j}(mL+\sum_{l=1}^{q}{k_lA_l}+A_{j+1})  \\
 &\to & \shf\otimes\OO_{Y_{j+1}}(mL+\sum_{l=1}^{q}{k_lA_l}+A_{j+1}) \to 0 \nonumber
 \end{eqnarray}
for all $0\leq j\leq q-1$, all $m$, and all $k_1,\dots,k_q\geq 0$. Here  $Y_j\deq E_1\cap\dots\cap E_j$ for all $1\leq j\leq q$, for the sake of completeness set $Y_0\deq X$.  Take a look at the last one of these. 

Fujita's vanishing theorem on $Y_q=E_1\cap\dots\cap E_q$ applied to the ample divisor $L|_{Y_q}$  gives that 
\[
 \HHv{i}{Y_q}{\shf\otimes\OO_{Y_q}(mL+\sum_{l=1}^{q}{k_lA_l})} \equ 0 
\]
for all $i\geq 1$, $m\geq m(\shf,L,A_1,\dots,A_q,Y_q)$ and all $k_1,\dots,k_q\geq 0$.

This implies that the groups on the sides of the exact sequence 
\begin{eqnarray*}
 && \HHv{i-1}{Y_q}{\shf\otimes\OO_{Y_q}(mL+\sum_{l=1}^{q}{k_lA_l}+A_q)} \lra \HHv{i}{Y_{q-1}}{\shf\otimes\OO_{Y_{q-1}}(mL+\sum_{l=1}^{q}{k_lA_l})}  \\ && \lra \HHv{i}{Y_{q-1}}{\shf\otimes\OO_{Y_{q-1}}(mL+\sum_{l=1}^{q}{k_lA_l}+A_q)} \lra \HHv{i}{Y_q}{\shf\otimes\OO_{Y_q}(mL+\sum_{l=1}^{q}{k_lA_l}+A_q)}
\end{eqnarray*}
vanish for $i\geq 2$,  $m\geq m(\shf,L,A_1,\dots,A_q,Y_q)$ and $k_1,\dots,k_q\geq 0$. Consequently,
\begin{eqnarray*}
\HHv{i}{Y_{q-1}}{\shf\otimes\OO_{Y_{q-1}}(mL+\sum_{l=1}^{q}{k_lA_l})} & \equ & \HHv{i}{Y_{q-1}}{\shf\otimes\OO_{Y_{q-1}}(mL+\sum_{l=1}^{q}{k_lA_l}+kA_q)}
\end{eqnarray*}
for all $i\geq 2$, $m\geq m(\shf,L,A_1,\dots,A_q,Y_q)$, and $k_1,\dots,k_q\geq 0$. Then
\[
 \HHv{i}{Y_{q-1}}{\shf\otimes\OO_{Y_{q-1}}(mL+\sum_{l=1}^{q}{k_lA_l}+kA_q)} \equ 0
\]
follows for all $k\geq 0$ from Serre vanishing applied to the  ample divisor $A_q|_{Y_{q-1}}$. By the semicontinuity theorem and the general choice of the $E_j$'s we can drop the dependence of $Y_q$.

Working backwards along the cohomology  sequences associated to the  sequences~\eqref{eqn:ses}, we obtain by  descending induction on $j$ that 
\[
 \HHv{i}{Y_j}{\shf\otimes\OO_{Y_j}(mL+\sum_{l=1}^{q}{k_lA_l})} \equ 0
\]
for $i> q-j$,$m\gg 0$, and all $k_1,\dots,k_q\geq 0$. This gives the required result when $j=0$.
\end{proof}

\begin{rmk}
We point out that the proof works under the less restrictive assumption that  $A_i$ is  ample, globally generated, and not composed of a pencil for all $1\leq i\leq q$.
\end{rmk}

The next step is to connect up with $q$-ampleness. Partial positivity was studied  in the form of uniform $q$-ampleness  in \cite{DPS}, where among many other achievements it was established that uniform  $q$-ampleness respects numerical equivalence of Cartier divisors. 
For the sake of completeness we briefly recall Totaro's main result on $q$-ample line bundles.

\begin{thm}\cite[Theorem 8.1]{Totaro}
Let $X$ be a projective scheme over a field of characteristic zero, $A$ a very ample divisor  on $X$, $0\leq q\leq n=\dim X$ an integer. Then there exists a natural number $m_0$ such that for all Cartier divisors $L$ on $X$  the following properties are equivalent.
\begin{enumerate}
 \item There exists a natural number $n_0$ such that $\HH{i}{X}{\OO_X(n_0L-jA)}=0$ for all $i>q$ and $1\leq j\leq m_0$.
\item ($L$ is naively $q$-ample) For every coherent sheaf $\shf$ on $X$ there exists an integer $m(L,\shf)$ such that $\HH{i}{X}{\shf\otimes\OO_X(mL)}=0$ for all $i>q$ and $m\geq m(L,\shf)$.
\item ($L$ is uniformly $q$-ample) There exists a constant $\lambda>0$ such that for all $i>q$,$j>0$  and  $\tfrac{m}{j}\geq \lambda$
the cohomology groups $\HH{i}{X}{\OO_X(mL-jA)}$ vanish. 
\end{enumerate}
 \end{thm}

This first consequence of Theorem~\ref{thm:A} is the following  claim, which was also proved by Demailly--Peternell--Schneider  under the more general assumption that $L$ is $(n-q)$-flag ample (see \cite[Definition 3.1]{DPS}). Their proof however requires considerably more effort. 

\begin{cor}
With notation as above, if $L|_{E_1\cap\dots\cap E_q}$ is ample for general $E_j\in |A_j|$, then $L$ is $q$-ample.
\end{cor}

The above vanishing result provides a birational variant for the higher asymptotic cohomology $\hhat{i}{X}{L}$ of $L$. We remind that
\[
 \hhat{i}{X}{L} \deq \limsup_{m} \frac{\hh{i}{X}{\OO_X(mL)}}{m^n/n!}\ ;
\]
note that $\widehat{h}^0(X,L)=\vol{X}{L}$ holds by definition. For properties of higher asymptotic cohomology  the reader is  referred to \cite{dFKL,ACF}, or Demailly's paper \cite{Dem} in the analytic setting. 

\begin{cor}\label{cor:asy}
Let $X$ be an irreducible projective variety, $L$ a Cartier divisor on $X$. Assume that there exists a proper birational morphism $\pi:Y\to X$, a natural number $q$, and very ample divisors $A_1,\dots ,A_q$ on $Y$ such that $\pi^*L|_{E_1\cap\dots\cap E_q}$ is ample for general elements $E_i\in |A_i|$, for all  $1\leq i\leq q$. Then  
\[
 \hhat{i}{X}{L} \equ 0 \ \ \text{ for $i>q$.}
\]
\end{cor}
\begin{proof}
As a consequence of Theorem~\ref{thm:A}, one has $\HH{i}{Y}{\pi^*\OO_X(mL)}=0$ for all $i>q$ and $m\gg 0$. This gives $\hhat{i}{Y}{\pi^*L}=0$ for all $i>q$. By the birational invariance of  asymptotic cohomology \cite[Corollary 2.10]{ACF}
\[
\hhat{i}{X}{L} \equ \hhat{i}{Y}{\pi^*L} \equ 0 \ \ \text{for all $i>q$.}
\]
\end{proof}

We move on to building a connection to the interior structure of the N\'eron--Severi space. For a Cartier divisor $L$
\[
q(L) \deq \min\st{q\in\NN\,|\, \text{$L$ is $q$-ample}} 
\]
is an interesting numerical invariant, which  was probably first defined in \cite[Definition 1.1]{DPS}.  To put it into perspective, let us briefly recall some other ways of  expressing partial ampleness associated to big divisors (these were discussed in an earlier version of \cite{dFKL}). 
\begin{align*} 
a(L;A_1,\dots,A_n) \ &\deq \  \min \{ k \mid\text{$L|_{E_1\cap\dots\cap E_k}$ is ample for very general $E_i \in
|A_i|$}\big \}, \\
b(L) \ &\deq \ \dim \Bplus(L), \\ 
 c(L) \ &\deq \ \max \big \{ i \mid \text{$\^h^i$ is not identically zero in any neighborhood of $[L]\in N^1(X)_\R$} \big \},
\end{align*}
where  $A_1,\dots,A_n$ are  very ample divisors on $X$. The minimum of all  $a(L;A_1,\dots,A_n)$ (over all sequences of very ample divisors of length $n=\dim X$) is  closely related to $\sigma_+(L)$ defined in \cite{DPS}.

The quantities  $q(L)$, $a(L,A)$, $b(L)$,  and $c(L)$ depend only on the numerical equivalence class of $L$, and make good sense for $\QQ$-divisors as well. They all express how far a given divisor is from being ample, with smaller numbers corresponding to more positivity.  

\begin{cor}\label{cor:ineq}
With notation as above, 
\[
 c(L) \dleq \ q(L) \dleq a(L;A_1,\dots,A_n) \dleq b(L)
\]
for all sequences of very ample divisors $A_1,\dots,A_n$.
\end{cor}
\begin{proof}
The first inequality comes from observing  the definition of $\^h^i$ and the openness of $q$-ampleness in the N\'eron--Severi space. The second one is \cite[3.4]{DPS}, at the same time, it is immediate from Theorem~\ref{thm:A}.

The inequality $a(L,A) \le b(L)$ comes from  the observation that the
restriction of a Cartier divisor to a general very ample divisor strictly
reduces  $\dim \Bplus(L)$, and the fact that a
divisor with empty augmented base locus is ample. 
\end{proof}

As it was noticed on \cite[p. 167.]{DPS}, one does not have equality in  $q(L)\deq a(L;A_1,\dots,A_n)$ in general. Here we present another simple example (borrowed again from an early version of \cite{dFKL}) exhibiting this property. More precisely, we give an example of a divisor $L$ that is $1$-ample, and a very ample divisor $A$ such that $L|_E$ is not ample for general $E\in|A|$.

\begin{eg}\label{eg:c(L)<a(L,A)<b(L)} Let $X = \F_1 \times \P^1$, where $\F_1$ is the blow-up of $\P^2$ at a point, and denote by $p : X \to
\F_1$ and $q : X \to \P^1$ the two projections. Let $E \subset \F_1$ be
the exceptional curve of the blow-up,  $F \subset \F_1$ be a fiber of the ruling, and let
$H \subset \P^1$ be a point. We consider the divisors
$$ L = p^*(\lambda E + F) + q^*H \ \text{ and } \ A = p^*(E + \mu F) + q^*H \quad
\text{for some} \quad \lambda, \mu \in\ZZ_{\ge 2}.
$$ 
Note that $A$ is very ample and $L$ is big. The stable  base locus of $L$ coincides with its augmented base locus and is equal to
$B \deq p^{-1}(E)$. In particular $b(L) = 2$.

On the other hand, the K\"unneth's formula for asymptotic cohomology (see
\cite[Remark~2.14]{ACF}), and the fact that $L$ is not ample imply that
$c(L) = 1$. Fix a general element $Y \in |A|$ cutting out a smooth
divisor $D$ on $B$. Note that
$\OO_B(D) \cong \OO_{\P^1\times\P^1}(\mu-1,1)$ via the isomorphism $B = E
\times \P^1 \cong \P^1 \times \P^1$. Therefore, since $D$ is smooth, it
must be irreducible; moreover,
$p$ induces an isomorphism $D \cong \P^1$.
We observe that the base locus
of $L|_Y$ is contained in the restriction of the base locus of $L$, hence
in $D$, and that $\OO_D(L|_Y) \cong \OO_{\P^1}(\mu - \lambda)$. We conclude that
$$ a(L,A) \ = \
\begin{cases} 2 = b(L) &\text{if $\lambda \ge \mu$}, \\ 1 = c(L) &\text{if $\lambda
< \mu$}.
\end{cases}
$$
For $\lambda <\mu$, we have  $1=c(L)\leq q(L)\leq a(L;A)=1$, therefore $q(L)=1$. On the other hand  there exist very ample divisors on $X$ such that $L|_A$ is not ample.
\end{eg}

  Totaro asks in \cite[Question 12.1]{Totaro} whether  $c(L)=q(L)$ always holds.  As a result of the discussion so far, we can see that $q(L)$ and $c(L)$ behave much the same way in relation to $a(L;A_1,\dots,A_n)$, which furnishes some evidence that the answer to Totaro's question is affirmative.

We give a reformulation of these concepts in terms of cones of divisors. This leads to a  connection with movable curves as defined in  \cite{BDPP}.   In what follows $A_1,\dots,A_q$  will without exception denote a sequence of very ample divisors. 

\begin{defi} For a natural number $q$ and very ample divisors $A_1,\dots, A_q$ set
\[
 \shc_{A_1,\dots,A_q} \deq \st{\alpha\in N^1(X)_\RR \,|\,\,\, \alpha|_{E_1\cap\dots\cap E_q} \text{ ample for $E_i\in |A_i|$ general }}\ .
\]
In addition, let $\Amp_q(X)$ denote the open (but not necessarily convex) cone of $q$-ample divisor classes.
\end{defi}

It is immediate that $\shc_{A_1,\dots,A_n}\subseteq N^1(X)_\RR$ is a convex  cone, it is also open by \cite[Section 9]{Totaro}; for simplicity we set $\shc_{\emptyset}=\Amp(X)$.

\begin{rmk}
In the important special case $q=n-1$ a general codimension $n-1$ complete intersection is an irreducible curve, and  one has 
\[
 \shc_{A_1,\dots,A_{n-1}} \equ \st{\alpha\in N^1(X)_\RR\,|\, (\alpha\cdot A_1\cdot\dots\cdot A_{n-1})>0}\ .
\]
\end{rmk}

Corollary~\ref{cor:ineq}  can then be rephrased in the following way.

\begin{cor}\label{cor:containment}
\[
 \bigcup_{A_1,\dots,A_q} \shc_{A_1,\dots,A_q} \dsubseteq \Amp_q(X) 
\]
\end{cor}

\begin{rmk}
The question arises naturally  whether the two sides of Corollary~\ref{cor:containment} are equal in general. This is quickly seen to be true on surfaces. By \cite[Theorem 9.1]{Totaro} $L$ is $1$-ample if and only if $-L$ is not pseudoeffective, which  is equivalent to the existence of a very ample divisor $A$ for which $(-L\cdot A)<0$. This latter holds precisely when  $L|_E$ is ample  for $E\in |A|$ general.

In general one obstruction that is easy to  foresee is the existence of  movable  curves on $X$ that are not in the boundary of the cone spanned by complete intersection curves. 
\end{rmk}

\begin{prop}\label{prop:equ}
Let $X$ be an irreducible projective variety of dimension $n$. Then 
\[
 \bigcup_{A_1,\dots,A_{n-1}} \shc_{A_1,\dots,A_{n-1}} \equ \Amp_{n-1}(X) 
\]
exactly if  every movable curve is the limit of elements of the convex cone spanned by complete intersection curves coming from very ample divisors. 
\end{prop}
\begin{proof}
For a Cartier divisor $L$ on $X$, $L|_{E_1\cap\dots\cap E_{n-1}}$ is ample if and only if  $(L\cdot A_1\cdot\dots\cdot A_{n-1})>0$. Consequently, 
$L\in\shc_{A_1,\dots,A_{n-1}}$ if and only if $(L\cdot A_1\cdot\dots\cdot A_{n-1})>0$. Consequently,  $L\in \cup_{A_1,\dots,A_{n-1}} \shc_{A_1,\dots,A_{n-1}}$ holds exactly  if there exists a complete intersection curve coming from very ample divisors intersected by $L$ positively.

By \cite[Theorem 9.1]{Totaro}, a  line bundle is $(n-1)$-ample precisely if  $-L$ is not pseudo-effective. The equality of the Proposition is equivalent  via \cite{BDPP} to the property that  a divisor intersecting every complete intersection curve of very ample divisors positively necessarily  intersects every movable curve positively. This happens precisely if  the cone spanned by complete intersection curves on $X$ is  dense in the cone of moving curves. 
\end{proof}

\begin{eg} In \cite[Example 3.2.4]{Neu} Neumann constructs a smooth projective threefold $X$ on which the cone spanned by complete intersection curves is not dense in the movable cone. The space he constructs is a double blow-up of $\PP^3$: first one blows up a line in $\PP^3$, then a point on the exceptional divisor. The work \cite{Neu} gives all the details. By Proposition~\ref{prop:equ}
\[
 \bigcup_{A_1,\dots,A_q} \shc_{A_1,\dots,A_{n-1}} \,\subsetneq\, \Amp_{n-1}(X)\ .
\]
\end{eg}

\begin{question}
Let $X$ be an irreducible projective variety. Under what condition does
\[
 \bigcup_{A_1,\dots,A_q} \shc_{A_1,\dots,A_q} \equ \Amp_q(X) 
\]
hold for all $0\leq q\leq n-1$?
\end{question}

\section{A Kawamata--Viehweg type vanishing for non-pseudo-effective divisors}

Independently of the discussion so far, we show that the ideas leading to Theorem~\ref{thm:A} also provide a partial vanishing theorem for adjoint divisors $K_X+L$  where $L$ is not necessarily  pseudo-effective. 

It has been common knowledge that cohomology groups of big line bundles tend to vanish  in degrees roughly  above the dimension of the stable base locus (see \cite[Proposition 2.15]{ACF} for an early example). Matsumura  in  \cite[Theorem 1.6]{Mats} proved that 
\[
 \HH{i}{X}{\OO_X(K_X+L)} \equ 0 \ \ \text{for $i> \dim B_{-}(L)$}
\]
for a big line bundle $L$.

In Theorem~\ref{thm:B} we present a variant which works without the bigness assumption, and in addition provides vanishing  in a  wider range of degrees  thanks  to the fact that $a(L,A)$ can be strictly  smaller than the dimension of the stable base locus  (see Example~\ref{eg:c(L)<a(L,A)<b(L)}). 

The proof of Theorem~\ref{thm:B}  requires a certain resolution defined in \cite[Section 4]{ACF}. Let $D$ be an arbitrary integral Cartier divisor, $A$ a very ample Cartier divisor on an irreducible projective variety $X$ of dimension $n$. Upon choosing general elements $E_1,\dots ,E_r\in |A|$, one obtains an exact sequence 
\begin{eqnarray}\label{eqn:res}
&& 0 \lra \OO_X(D) \lra \OO_X(D+rA) \lra \bigoplus_{i=1}^{r}\OO_{E_i}(D+rA) \lra \\
&&  \lra \bigoplus_{1\leq i_1<i_2\leq r}\OO_{E_{i_1}\cap E_{i_2}}(D+rA) \lra \dots\lra \bigoplus_{1\leq i_1<i_2<\dots<i_n\leq r}\OO_{E_{i_1}\cap\dots\cap E_{i_n}}(D+rA) \lra 0\ . \nonumber
\end{eqnarray}
Given a coherent sheaf $\shf$ one can also assume that  the sequence \eqnref{eqn:res} remains exact after tensoring by  $\shf$ by the general position of the effective divisors $E_i$.

Although strictly speaking it would not be necessary, it helps in the book-keeping process to chop up the above resolution into short exact sequences 
\begin{eqnarray*}
&& 0 \to \shf\otimes\OO_X(D) \to \shf\otimes\OO_X(D+rA) \to \shc_1 \to  0 \\
&& 0 \to \shc_1 \to \bigoplus_{i=1}^{r}\shf\otimes\OO_{E_i}(D+rA) \to \shc_2 \to 0 \\
&& \vdots \\
&& 0 \to \shc_{n-1} \to  \bigoplus_{1\leq i_1<i_2<\dots<i_{n-1}\leq r}\shf\otimes\OO_{E_{i_1}\cap\dots\cap E_{i_{n-1}}}(D+rA)  \\
&& \to  \bigoplus_{1\leq i_1<i_2<\dots<i_n\leq r}\shf\otimes\OO_{E_{i_1}\cap\dots\cap E_{i_n}}(D+rA) \to 0\ .
\end{eqnarray*}

\begin{thm}\label{thm:B}
 Let $X$ be a smooth projective variety, $L$ a divisor, $A$ a very ample divisor on $X$. If $L|_{E_1\cap\dots\cap E_q}$ is big and nef for a general choice of $E_1,\dots,E_q\in |A|$, then 
\[
 \HH{i}{X}{\OO_X(K_X+L)} \equ 0 \ \ \ \text{for $i>q$.}
\]
\end{thm}

\begin{proof}
We will prove the statement by induction on the codimension of the complete intersections we restrict to;  the case $q=0$ is the  Kawamata--Viehweg vanishing theorem.  Let $E_1,\dots,E_q\in |A|$ be elements such that the intersection of any combination of them is smooth of the expected dimension, and irreducible when it has positive dimension. As the $E_i$'s are assumed to be general, this can clearly be done via  the base-point free Bertini theorem, which works for $\dim X\geq 2$. In the remaining cases (when $\dim X\leq 1$) the statement of the proposition is immediate.

Consider the exact sequence \eqnref{eqn:res} with $D=K_X+L+mA$ and  $r=q$.
First we show that it suffices to verify
\[
 \HH{i}{X}{\shc_1^{(m)}} \equ 0 \ \ \ \text{for all $m\geq 0$ and $i>q-1$},
\]
where the upper index of $\shc$ is used  to emphasize the explicit dependence on $m$.  Grant this for the moment, and  see how this helps up to prove the statement of the proposition. 

Take  the following part of the long exact sequence associated to the first piece above
\begin{eqnarray*}
 \HH{i-1}{X}{\shc_1^{(m)}} & \to &  \HH{i}{X}{\OO_X(K_X+L+mA)} \\
&& \to \HH{i}{X}{\OO_X(K_X+L+(m+q)A)} \to \HH{i}{X}{\shc_1^{(m)}}\ .
\end{eqnarray*}
By assumption the cohomology groups on the two sides vanish for all $m$ whenever $i>q$, hence 
\[
\HH{i}{X}{\OO_X(K_X+L+mA)} \simeq \HH{i}{X}{\OO_X(K_X+L+(m+q)A)}  \ \ \ \text{for all $m\geq 0$ and $i>q$.} 
\]
These groups are zero however for $m$ sufficiently large by Serre vanishing, hence 
\[
\HH{i}{X}{\OO_X(K_X+L)} \equ 0 \ \ \ \text{ for all $i>q$,}
\]
as we wanted.

As for  the vanishing of the cohomology groups $\HH{i}{X}{\shc_1^{(m)}}$ for $m\geq 0$ and $i>q-1$, it is quickly checked inductively.
Observe that for all $1\leq j\leq q$ we have
\[
 K_X+L+(m+q)A|_{E_1\cap\dots\cap E_j} \equ K_{E_1\cap\dots\cap E_j} + (L+(m+(q-j))A)|_{E_1\cap\dots\cap E_j}
\]
by adjunction, and $(L+(m+(q-j))A)|_{E_1\cap\dots\cap E_j}$  becomes ample when restricted to the intersection with $E_{j+1}\cap\dots\cap E_q$. 
Induction gives
\[
 \HH{i}{E_1\cap\dots\cap E_j}{K_X+L+(m+q)A|_{E_1\cap\dots\cap E_j}} \equ 0 \ \ \ \text{for all $m\geq 0$ and $i>q-j$}\ .
\]
By chasing through the appropriate long exact sequences, we arrive at 
\[
 \HH{i}{X}{\shc_j^{(m)}} \equ 0 \ \ \ \text{for all $m\geq 0$ and $i>q-j$.}
\]
For $j=1$ this is the required vanishing. 
\end{proof}

\begin{cor}\label{cor:birational vanishing}
Let $X$ be an irreducible projective variety, $L$ a line bundle, $A$ a very ample line bundle on $X$, $q \geq 0$ such that $L|_{E_1\cap\dots\cap E_q}$ is big and nef for $E_i\in |A|$ general ($1\leq i\leq q$). If $\pi:Y\to X$ is a proper birational morphism from a smooth variety, $B$ a nef divisor on $Y$, then
\[
 \HH{i}{Y}{\OO_Y(K_Y+\pi^*L+B)} \equ 0
\]
for all $i>q$.
\end{cor}
\begin{proof}
This is in fact a corollary of the proof of Theorem~\ref{thm:B}. We point out the necessary modifications. Assuming $\dim X\geq 2$, Lemma~\ref{lem:pullback} makes sure that we can consider restrictions  to intersections of  general elements of $|\pi^*A|$ just as we did in the proof of Theorem~\ref{thm:B}; we also obtain that the generic restriction  $\pi^*L|_{E_1'\cap\dots\cap E_q'}$ is still big and nef for  $E_i'\in |\pi^*A|$ general. 

Next, run the proof on $Y$, with $D=K_Y+\pi^*L+D+m\pi^*A$, and $r=q$. The task that remains is to show that the cohomology groups
\[
\HH{i}{Y}{\OO_Y(K_Y+\pi^*L+D+m\pi^*A)} \simeq \HH{i}{Y}{\OO_X(K_Y+\pi^*L+D+(m+q)\pi^*A)}  
\]
vanish for all $m\geq 0$ and $i>q$. By the given isomorphisms, it suffices to prove this for $m\gg 0$. Serre vanishing no longer applies, since $\pi^*A$ is only big and nef; luckily we can use the classical Kawamata--Viehweg theorem  to our advantage. Namely, observe that 
\[
 \pi^*L+D+m\pi^*A \equ \pi^*(L+m_0A)+D+(m-m_0)\pi^*A
\]
for all integers $m,m0$  with $m\geq m_0$. If $m_0$ is suitably large then $L+m_0A$ itself is ample, therefore $\pi^*(L+m_0A)+D+(m-m_0)\pi^*A$ is big and nef, and the  required vanishing follows. 
\end{proof}

\begin{lem}\label{lem:pullback}
 Let $\pi:Y\to X$ a proper birational morphism of irreducible projective varieties of dimension $n\geq 2$, $L$ a Cartier divisor, $A$ a very ample Cartier divisor on $X$. If $L|_{E_1\cap\dots\cap E_k}$ is big and nef for some $k\geq 1$ and general elements $E_1,\dots,E_k$, then $\pi_*L|_{E_1'\cap\dots\cap E_k'}$ is big and nef  for  the same integer $k$, and $E_1',\dots, E_k'$ general elements from $|\pi^*A|$.  
\end{lem}
\begin{proof}
 As $\pi^*A$ is big and globally generated and $\dim Y\geq 2$, a general element $E'\in |\pi^*A|$ maps to a general element of $|A|$ by the base-point free Bertini theorem. Moreover, by the same token,  the intersection $E_1'\cap\dots\cap E_k'$ of general elements $|\pi^*A|$ is irreducible, and  $\pi|_{E_1'\cap\dots\cap E_k'}$ is a proper birational morphism onto its image, which is the intersection of $k$ general elements of $|A|$, say $E_1\cap\dots\cap E_k$ (we tacitly assume that $E_i'$ maps to $E_i$ under $\pi$). 

Consequently, $\pi^*(L|_{E_1\cap\dots\cap E_k})$ is a big and nef divisor on $E_1'\cap\dots\cap E_k'$. However,
\[
 \pi^*(L|_{E_1\cap\dots\cap E_k}) \equ (\pi^*L)|_{E_1'\cap\dots\cap E_k'}\ , 
\]
hence the latter is big and nef as we wanted. 
\end{proof}

\section{Base loci on schemes}

Here we treat base loci of line bundles on arbitrary schemes. Although a large part of what we do is straightforward, the topic has not been investigated much so far, and there is no  suitable reference available. In the course of this section $X$ is an arbitrary scheme unless otherwise mentioned. 

It is customary (see \cite[Section 1.1.B]{PAG} for example) to define  the base ideal sheaf  of a Cartier divisor  $L$ on a complete  algebraic scheme  to be 
\[
\mathfrak{b}(L) \deq  \im \left(  \HH{0}{X}{\OO_X(L)} \otimes_\CC \OO_X(-L) \stackrel{\text{eval}_L}{\lra} \OO_X  \right)\ .
\]
One then sets $\Bs(L)$ to be the closed subscheme of $X$ given by $\mathfrak{b}(L)$, and 
\[
\B(L) \deq  \bigcap_{m=1}^{\infty} \Bs(mL)_{\text{red}} \dsubseteq X
\]
as a closed subset.  We can nevertheless define the base locus of an invertible sheaf in full generality.

\begin{defi}\label{defi:general}
Let $X$ be a scheme, $\shl$ an invertible sheaf on $X$. Let $\shf_\shl$ denote the quasi-coherent subsheaf of $\shl$ generated by $\HH{0}{X}{\shl}$. With this notation set 
\[
\mathfrak{b}(\shl) \deq \ann_{\OO_X} (\shl/\shf_\shl)\ , 
\]
define $\Bs(\shl)$ to be the closed subscheme corresponding to $\mathfrak{b}(\shl)$, and let 
\[
\B(\shl) \deq  \bigcap_{m=1}^{\infty} \Bs(\shl^{\otimes m})_{\text{red}} \dsubseteq X
\]
as a closed subset of the topological space associated to $X$.
\end{defi}

It is immediate that we recover the usual definition in the case $X$ is complete and algebraic.

\begin{lem}\label{lem:bsl}
Let $X,Y$ be  schemes, $f:Y\to X$ a map of schemes, $\shl$ an invertible sheaf on $X$. Then
\[
 \mathfrak{b} (\shl) \cdot \OO_Y \dsubseteq \mathfrak{b} (f^*\shl)\ .
\]
In particular, if $Y\subseteq X$ is a closed subscheme, then  $\Bs(\shl|_Y) \subseteq \Bs(\shl)\cap Y$.
\end{lem}
\begin{proof}\footnote{I would like to thank Brian Conrad for simplifying a previous argument considerably and pointing out the right degree of generality here and in Definition~\ref{defi:general}.}
Observe that to any quasi-coherent subsheaf $\shf$  of an invertible sheaf $\shl$ one can associate a (quasi-coherent) sheaf of ideals
\[
 \shi(\shf) \deq \ann_{\OO_X} \shl/\shf \ . 
\]
The definition implies that $\shi(\shf)\subseteq \shi(\shf')$ whenever $\shf\subseteq\shf'$.

Considering the map $\HH{0}{X}{\shl}\to\HH{0}{Y}{f^*\shl}$ obtained by pulling back sections, one observes that the map 
\[
 f^*(\shf_\shl) \lra f^*\shl
\]
factors through the sheaf of modules $\shf_{f^*\shl}$. Consequently, 
\[
\mathfrak{b}(\shl)\cdot\OO_Y \equ  \shi(\im f^*(\shf_\shl)) \dsubseteq \shi(\shf_{f^*\shl}) \equ \mathfrak{b}(f^*\shl)\  
\] 
as we wanted.
\end{proof}

\begin{cor}\label{cor:bsl} 
In the  situation of Lemma~\ref{lem:bsl} one has 
\[
 \dim \B(\shl|_Y) \dleq  \dim \B(\shl)\ ,
\]
which   immediately extends to the case of  $\QQ$-Cartier divisors.
\end{cor}

\begin{rmk} We point out that the proofs of both \cite[Example 1.1.9]{PAG} and \cite[Proposition 2.1.21]{PAG}  go through unchanged when $X$ is an arbitrary  scheme. This holds since both proofs depend only on the property
\[
 \mathfrak{b}(\shl^{\otimes m}) \cdot \mathfrak{b}(\shl^{\otimes k}) \dsubseteq \mathfrak{b}(\shl^{\otimes(m+k)}) \ \text{for all $m,k\geq 1$}\ ,
\]
which follows from the fact that one can multiply global sections.  Therefore we obtain that $\B(\shl)$ is the unique minimal member of the collection of closed subsets $\st{\Bs(\shl^{\otimes m})_{\text{red}}\,|\, m\geq 1}$. Moreover, just as in the reduced and irreducible case there exists $m_0\in \NN$ with the property that $\B(\shl) = \Bs(\shl^{\otimes pm_0})_{\text{red}}$ for all natural numbers $p$.

As a consequence, $\B(\shl) = \B(\shl^{\otimes m})$ for all positive integers $m$, and we are allowed to define the stable base locus for $\QQ$-Cartier divisors by taking the stable base locus of a Cartier multiple.  
\end{rmk}

Following \cite[Remark 1.3]{AIBL}, we define the augmented base locus of a $\QQ$-Cartier divisor  $L$ via
\[
 \Bplus (L) \deq \bigcap_{A} \B(L-A)	
\]
where $A$ runs through all ample $\QQ$-Cartier divisors. 

\begin{rmk}\label{rmk:augm}
Let us assume that $X$ is  complete and algebraic over $\CC$. Arguing as in the proof of \cite[Proposition 1.5]{AIBL} we can see that for a given $\QQ$-divisor $L$ one can always find an $\epsilon>0$ 
such that 
\[
 \Bplus (L) \equ \B(L-A)
\]
for any ample $\QQ$-divisor with $\| A \| <\epsilon$ (with respect to an arbitrary norm on the N\'eron--Severi space). Exploiting Corollary~\ref{cor:bsl}  this implies that 
\[
 \dim \Bplus(L|_Y) \dleq  \dim \Bplus(L)
\]
for any closed subscheme $Y$ in $X$.
\end{rmk}

The main contribution of this section is  a generalization of the fact that any coherent sheaf becomes globally generated after twisting with a high enough multiple of an ample line bundle. 

\begin{prop}\label{prop:resolution}
Let $X$ be an irreducible projective variety, $L$ a Cartier divisor on $X$.  Then $\Bplus(L)$ is the smallest  subset of $X$ such that for all coherent sheaves $\shf$ on $X$  there exists a possibly infinite sequence of sheaves of the form
\[
\dots \to  \bigoplus_{i=1}^{r_i} \OO_X(-m_iL) \to\dots\to \bigoplus_{i=1}^{r_1} \OO_X(-m_1L) \to \shf\ ,
\]
which is exact off $\Bplus(L)$.
\end{prop}
\begin{proof}
If $L$ is a non-big divisor, then $\Bplus(L)=X$, and the statement is obviously true. Hence we can assume without loss of generality that $L$ is big.

First we prove the following claim: let $\shf$ be an arbitrary coherent sheaf on $X$; then there exist  positive integers $r$,$m$, and a map of sheaves
\begin{equation}\label{eqn:first}
 \bigoplus_{i=1}^{r} \OO_X(-mL) \stackrel{\phi}{\lra} \shf\ , 
\end{equation}
which is surjective away from $\Bplus (L)$.  

Fix an arbitrary ample divisor $A$ on $X$. The sheaves $\shf\otimes\OO_X(m'A)$ are globally generated for $m'$ sufficiently large. According to \cite[Proposition 1.5]{AIBL} $\Bplus(L)=\B (L-\epsilon A)$ for any $\epsilon>0$ small enough. Pick such an $\epsilon$, set $L'\deq L-\epsilon A$ and let $m\gg 0$ be a positive integer such that $m'\deq m\epsilon$ is an integer, and  
\[
 \Bs(mL') \equ \B (mL') \equ \Bplus (L)\ .
\]
By picking $m$ large enough, we can in addition assume that   $\shf\otimes\OO_X(m'A)$ is globally generated. As a consequence,
\[
 \shf\otimes\OO_X(m'A)\otimes\OO_X(mL') \simeq  \shf\otimes\OO_X(m'A+mL') 
\]
is globally generated away from $\Bs(mL')=\Bplus (L)$. On the other hand 
\[
mL'+m'A \equ m(L-\epsilon A)+(m\epsilon)A \equ mL\ ,
\]
hence we have found $m\gg 0$ such that $\shf\otimes\OO_X(mL)$ is globally generated away from $\Bplus (L)$. Thanks to the map
\[
 \HH{0}{X}{\shf\otimes\OO_X(m'A)}\otimes \HH{0}{X}{\OO_X(mL')} \lra \HH{0}{X}{\shf\otimes\OO_X(mL)}
\]
one can find a finite set of sections giving rise to a map of sheaves
\[
 \bigoplus_{i=1}^{r}\OO_X \lra \shf\otimes\OO_X(mL)
\]
surjective away from $\Bplus (L)$. Tensoring by $\OO_X(-mL)$ gives the map in \eqnref{eqn:first}.

Next we will prove that $\Bplus(L)$ satisfies that property described in the Proposition. Let $\shg$ be the kernel of the map $\bigoplus_{i=1}^{r_1} \OO_X(-m_1L) \stackrel{\phi_1}{\lra} \shf$ coming from \eqnref{eqn:first}.
Applying \eqnref{eqn:first} to $\shg$ we obtain a map
\[
 \bigoplus_{i=1}^{r_2} \OO_X(-mL) \stackrel{\phi_2}{\lra} \shg
\]
surjective off $\Bplus (L)$, hence a two-term sequence 
\[
\bigoplus_{i=1}^{r_2} \OO_X(-m_2L) \to \bigoplus_{i=1}^{r_1} \OO_X(-m_1L) \to \shf 
\]
exact away from the closed subset $\Bplus (L)$. Continuing in this fashion we arrive at a possibly infinite sequence of the required type.

Last, if $x\in \Bplus(L)$ then for all $\epsilon\in\QQ^{\geq 0}$ and all $m\geq 1$ such that $m\epsilon\in\ZZ$, all global sections of $\OO_X(m(L-\epsilon A))$ vanish at $x$. By taking $\shf\deq \OO_X(A)$, 
\[
 \shf\otimes \OO_X(mL) \equ \OO_X(mL-A)
\]
will then have all global sections vanishing at $x\in X$. Therefore $\Bplus(L)$ is indeed the smallest subset of $X$  with the required property.
\end{proof}

\section{Generalized Fujita vanishing}

We prove  a generalization of Fujita's vanishing theorem (see \cite{Fuj} or \cite[Theorem 1.4.35]{PAG}) following  his  original line of thought. The main technical tools we will rely on are Proposition~\ref{prop:resolution} and  Corollary~\ref{cor:birational vanishing}.

\begin{thm}\label{thm:Fujita}
Let $X$ be a complex projective scheme, $L$ a Cartier divisor, $\shf$ a coherent sheaf on $X$. Then there exists a positive integer $m_0(L,\shf)$ such that 
\[
 \HH{i}{X}{\shf\otimes\OO_X(mL+D)} \equ 0
\]
 for all $i>\dim \Bplus (L)$, $m\geq m_0(L,\shf)$, and all nef divisors $D$ on $X$.
\end{thm}

First we  reduce the theorem  to the case of irreducible projective varieties.

\begin{lem}\label{lem:reduction}
With notation as above, if Theorem~\ref{thm:Fujita} holds when $X$ is reduced and irreducible, then it is true in general. 
\end{lem}
\begin{proof}
The proof of the fact a Cartier divisor on a complete scheme is ample if and only if it so when restricted to any irreducible component of the underlying reduced scheme (\cite[1.2.16]{PAG}) works here as well with a few modifications. For simplicity we will give the whole proof. First we show that  Theorem~\ref{thm:Fujita} holds for $X$ provided it is true on $X_\textrm{red}$.  To this end, let $X_{\textrm{red}}$ be $X$ with the reduced induced subscheme structure, $\shn$ the nilradical of $\OO_X$. 

Let now $\shg$ be a coherent sheaf on $X$. We will work with the filtration
\[
 \shg \dsup \shn\cdot\shg\dsup\dots\dsup\shn^r\cdot \shg \equ 0\ ,
\]
which gives rise to a bunch of short exact sequences
\[
\ses{\shn^{j+1}\cdot\shg}{\shn^j\cdot\shg}{\shn^j\cdot\shg/\shn^{j+1}\cdot\shg}\ . 
\]
The quotient sheaves $\shn^j\cdot\shg/\shn^{j+1}\cdot\shg$ are coherent $\OO_{\xred}$-modules, therefore the assumption that Theorem~\ref{thm:Fujita} holds on reduced schemes  implies
\begin{eqnarray*}
&& \HH{i}{X}{(\shn^j\cdot\shg/\shn^{j+1}\cdot\shg)\otimes \OO_X(mL+D)} \\ 
 && \equ  \HH{i}{\xred}{\shn^j\cdot\shg/\shn^{j+1}\cdot\shg\otimes\OO_{\xred}(mL+D)} \equ 0
\end{eqnarray*}
for all $i>\dim \Bplus (L|_{X_{\text{red}}})$, $m\gg 0$ and nef divisor $D$ on $X$. Applying the above vanishing to  the long exact sequence 
associated to 
\begin{eqnarray*}
&& 0 \to  {(\shn^{j+1}\cdot\shg)\otimes\OO_X(mL+D)}   \to  {(\shn^j\cdot\shg)\otimes\OO_X(mL+D)} \to \\ 
&&  \to {(\shn^j\cdot\shg/\shn^{j+1}\cdot\shg)\otimes\OO_X(mL+D)} \to 0 
\end{eqnarray*}
we see that $m\gg 0$ and $j\geq 0$
\[
 \HH{i}{X}{(\shn^{j+1}\cdot\shg)\otimes\OO_X(mL+D)} \equ \HH{i}{X}{\shn^j\cdot\shg)\otimes\OO_X(mL+D)}  
\]
whenever $i> \dim \Bplus (L|_{X_{\text{red}}})+1$, and
\[
 \HH{i}{X}{(\shn^{j+1}\cdot\shg)\otimes\OO_X(mL+D)} \twoheadrightarrow \HH{i}{X}{\shn^j\cdot\shg)\otimes\OO_X(mL+D)} 
\]
for $i=\dim \Bplus (L|_{X_{\text{red}}})+1$. According to Remark~\ref{lem:bsl}, $\dim \Bplus (L|_{X_{\text{red}}}) \leq\dim \Bplus (L)$, therefore the same 
vanishing and surjectivity results hold whenever $i>\dim \Bplus (L)+1$ and $i=\dim \Bplus (L)+1$, respectively. Descending induction on $j$ gives the required vanishing statement.

From now on we can and will assume without loss of generality that $X$ is reduced.  Write
\[
 X \equ X_1\cup\dots\cup X_r
\]
as the union of its irreducible components. In the short exact sequence
\[
 \ses{\II\cdot\shg}{\shg}{\shg/\II\cdot\shg}
\]
the left term is supported on $X_2\cup\dots\cup X_r$, while the right term is supported in $X_1$.  Induction on the number of irreducible components then tells us that
\begin{eqnarray*}
&&  \HH{i}{X}{(\II\cdot\shg)\otimes\OO_X(mL+D)} \equ \HH{i}{X_2\cup\dots\cup X_r}{(\II\cdot\shg)\otimes\OO_{X_2\cup\dots\cup X_r}(mL+D)} = 0
\end{eqnarray*}
and
\[
 \HH{i}{X}{(\shg/\II\cdot\shg)\otimes\OO_X(mL+D)} \equ \HH{i}{X_1}{(\shg/\II\cdot\shg)\otimes\OO_{X_1}(mL+D)} \equ 0
\]
for $i>\max\st{ \dim \Bplus (L|_{Y_1}),\dim \Bplus (L|_{Y_2\cup\dots\cup Y_r})}$, $m\gg 0$, and all nef divisors $D$ on $X$. From the associated cohomology long exact sequence we obtain that 
\[
 \HH{i}{X}{\shg\otimes\OO_X(mL+D)} \equ 0
\]
for all $i>\max\st{ \dim \Bplus (L|_{Y_1}),\dim \Bplus (L|_{Y_2\cup\dots\cup Y_r})}$, $m\gg 0$, and all nef divisors $D$ on $X$. We can conclude the proof by 
\[
\dim \Bplus (L) \dgeq \max\st{\dim \Bplus (L|_{Y_1}),\dim \Bplus (L|_{Y_2\cup\dots\cup Y_r})} \ .
\] 
\end{proof}

\begin{proof}(of Theorem~\ref{thm:Fujita})
By Lemma~\ref{lem:reduction}, we can assume that $X$ is reduced and irreducible. By induction on dimension we can also assume that the result is known for all sheaves supported on a proper subscheme of $X$. If $\dim \Bplus(L)\geq n$, then the Theorem holds for dimension reasons, hence we can assume that $\Bplus (L) < n$, equivalently, that $L$ is big.

According to Proposition~\ref{prop:resolution} the sheaf $\shf$ possesses  a possibly infinite 'resolution'
\[
\cdots \lra \oplus\OO_X(a_1L) \lra \oplus\OO_X(a_0L) \lra \shf\lra 0 
\]
whose cohomology sheaves are supported on $\Bplus (L)$. Therefore \cite[Proposition B.1.2]{PAG} and \cite[Remark B.1.4]{PAG}
imply that  one can assume that it suffices to find one integer $a\in \ZZ$ for which the theorem holds for $\OO_X(aL)$. 

Let $\mu:\tilde{X}\to X$ be a resolution of singularities, and set $\shk_X\deq \mu_*\OO_{\tilde{X}}(K_{\tilde{X}})$.
The divisor $L$ is big, therefore $\OO_{\tilde{X}}(\mu^*(aL)-K_{\tilde{X}})$ will have a section for $a$ sufficiently large.
This shows that  there exists an injection
\[
 u: \shk_X \hookrightarrow \OO_X(aL)
\]
of coherent sheaves on $X$.  We end up having reduced the theorem to the case when $\shf=\shk_X$. 

By Grauert--Riemenschneider vanishing we have 
\[
 R^j\mu_*\OO_{\tilde{X}}(K_{\tilde{X}}) \equ 0 \ \ \text{for all $j>0$.}
\]
Therefore
\[
 \HH{i}{X}{\shk_X\otimes\OO_X(aL+D)} \equ \HH{i}{\tilde{X}}{\OO_{\tilde{X}}(K_{\tilde{X}}+\mu^*(aL+D))}
\]
for all $i$. Consider the cohomology group on the right-hand side. By Corollary~\ref{cor:ineq}, the restriction $L|_{E_1\cap\dots\cap E_q}$ is ample for $q> \dim\Bplus (L)$. Then Corollary~\ref{cor:birational vanishing} implies 
\[
 \HH{i}{\tilde{X}}{\OO_{\tilde{X}}(K_{\tilde{X}}+\mu^*(aL)+\mu^*D))} \equ 0
\]
for all $i>\dim \Bplus (L)$.
\end{proof}

\end{document}